\documentclass{amsart}

\usepackage{mathrsfs, mathtools, amssymb}

\usepackage{dsfont}

\usepackage[paper=a4paper, margin=3cm]{geometry}

\usepackage[breaklinks]{hyperref}       
\usepackage{url}            
\usepackage{booktabs}       
\usepackage{amsfonts}       
\usepackage{nicefrac}       
\usepackage{microtype}      
\usepackage{lipsum}

\usepackage{amsmath,amssymb,amsthm,amscd}
\usepackage{amsrefs}
\usepackage{tikz,tikz-cd}
\usetikzlibrary{arrows,arrows.meta,calc}

\usepackage{graphicx}
\graphicspath{}
\DeclareGraphicsExtensions{.pdf,.png,.jpg,.jpeg}

\usepackage{yhmath}
\usepackage{enumitem}
\usepackage{breakurl}
\usepackage{subcaption}

\usepackage{tikz}
\usepackage{tikz-cd}
\usetikzlibrary{arrows}
\usepackage{epsfig}
\usepackage[all]{xy}
\usepackage{epstopdf}
\usepackage{framed}

\newcommand{\Z}{\mathbb Z}
\newcommand{\Q}{\mathbb Q}

\newcommand{\lb}{\lbrace}
\newcommand{\rb}{\rbrace}
\newcommand{\la}{\langle}
\newcommand{\ra}{\rangle}
\renewcommand{\phi}{\varphi}
\newcommand{\eps}{\varepsilon}

\DeclareMathOperator{\rk}{rk}

\DeclareMathOperator{\Homs}{\mathscr{H}\text{\kern -3pt {\calligra\large om}}\,}

\DeclareMathOperator{\str}{star}

\DeclareMathOperator{\GKM}{GKM}

\DeclareMathOperator{\spn}{span}

\renewcommand{\leq}{\leqslant}
\renewcommand{\geq}{\geqslant}

\theoremstyle{plain}
\newtheorem{thm}{Theorem}[section]
\newtheorem{lm}[thm]{Lemma}

\newtheorem{pr}[thm]{Proposition}

\theoremstyle{definition}

\newtheorem{rem}[thm]{Remark}
\newtheorem{ex}[thm]{Example}
\newtheorem{defn}[thm]{Definition}

\newcounter{thmMaincounter}
\newtheorem{thmM}[thmMaincounter]{Theorem}

\tikzcdset{arrow style=tikz, diagrams={>=stealth}}

\begin{document}
	
	\title[Epimorphisms and face ring quotients]{
		Equivariant cohomology epimorphisms and face ring quotients for Hamiltonian and complexity one GKM$_4$ manifolds
	}
	
	\author{Oliver Goertsches}
	\address[O.\,Goertsches]{Philipps-Universit\"at Marburg, Germany}
	\email{goertsch@mathematik.uni-marburg.de}
	
	\author{Grigory Solomadin}
	\address[G.\,Solomadin]{Philipps-Universit\"at Marburg, Germany}
	\email{grigory.solomadin@gmail.com}
	
	\begin{abstract}
		Given a GKM$_3$ action of a torus $K$ on a manifold $M$ with GKM graph $\Gamma$, we show that for any extension of $\Gamma$ to an abstract GKM graph the corresponding restriction map in equivariant graph cohomology is surjective. While the corresponding statement for extensions of actions is well-known, we observe that this graph-theoretical statement is false in the GKM$_2$ setting. As a corollary, we obtain a description of the equivariant cohomology ring of Hamiltonian and complexity one GKM$_4$ actions in terms of generators and relations.\\
	\end{abstract}

	

	\maketitle
	
	\section{Introduction}
	
	Given a closed orientable $T$-manifold $M$ with finite and nonempty fixed point set $M^{T}$, the equivariant cohomology $H^{*}_{T}(M;\Bbbk)$ is a free $H^{*}(BT;\Bbbk)$-module if and only if $H^{odd}(M;\Bbbk)=0$~\cite[Lemma 2.1]{ma-pa-06}.
    (In this paper, cohomology is singular equivariant cohomology, with coefficients $\Bbbk$ being either the rationals or the integers.)
	This property is called \textit{equivariant formality over $\Bbbk$} of the $T$-action on $M$. 
	In particular, for any $T$-action on $M$ and any subtorus $K\subset T$, assuming that both the $T$-action and its restriction to $K$ have finite fixed point set, equivariant formality is equivalent for the $T$-action and the restricted $K$-action.
	In this case, the induced map $H^*_{T}(M;\Bbbk)\to H^*_{K}(M;\Bbbk)$ is surjective. For rational coefficients, a more general statement was proven in~\cite[Theorem 5.7]{al-fr-pu-14}, relating the syzygy order of $H^*_T(M;\Q)$ to the surjectivity statement for subtori of certain dimensions. 
	
	An equivariantly formal $T$-action on a closed, orientable manifold $M$ is called a \textit{GKM action} (and $M$ is called a \textit{GKM manifold}), if the union of the one-dimensional orbits $M_{1}$ of the $T$-action is a finite union of two-dimensional $T$-invariant spheres, i.e., if the orbit space $M_1/T$ is homeomorphic to a graph.
	This graph, together with the $H^{2}(BT;\Bbbk)/\lb\pm 1\rb$-labeling of its edges by weights of the isotropy representations of the $T$-action in the fixed points is called the \textit{(unsigned) GKM graph} $\Gamma=\Gamma(M)$ of the $T$-action.
	The \textit{equivariant graph cohomology} $H^{*}_{T}(\Gamma;\Bbbk)$ (see \eqref{eq:gkmring} below) of the GKM graph $\Gamma$ is isomorphic to $H^{*}_{T}(M;\Bbbk)$ as a $H^{*}(BT;\Bbbk)$-module for any GKM manifold $M$ by exactness of the Chang-Skjelbred sequence~\cite{ch-sk-74,gkm-98}  (assuming connected stabilizers in case of integer coefficients).	
	
	The starting point of this note is the observation that the above surjectivity statement for the restriction map in equivariant cohomology is no longer true in the context of equivariant graph cohomology of abstract GKM graphs. More precisely, given a GKM action of a torus $K$ on a manifold $M$ with the GKM graph $\Gamma_K$, and an extension of $\Gamma_K$ to an abstract GKM graph $\Gamma_T$ with $T$-labeling (see Section~\ref{subsec:eqcohomextension} below), neither the natural restriction map in equivariant graph cohomology $p_{*}\colon H^*_T(\Gamma_T;\Bbbk)\to H^*_K(\Gamma_K;\Bbbk)$ is surjective nor $H^*_T(\Gamma_T;\Bbbk)$ is a free $H^{*}(BT)$-module, in general. We provide the respective counterexample in Section~\ref{sec:GKMcounter}. 
	Recall that a GKM manifold $M$ is called \textit{GKM$_{j}$} if any $i\leq j$ pairwise distinct isotropy weights at $T_{x} M$, $x\in M^{T}$, are linearly independent in $H^{2}(BT)$ (this condition is called \emph{$j$-independency} in~\cite{ay-ma-so-23}, respectively \emph{an action in $j$-general position} in~\cite{ay-ma-23}).
	
	\begin{thmM}\label{thmM:epi}
		In the situation above, for any GKM$_{3}$ $K$-manifold $M$, the map
		\[
		\begin{tikzcd}
		H^{*}_{T}(\Gamma_{T};\Bbbk)\arrow{r}{p_{*}} & H^{*}_{K}(\Gamma_{K};\Bbbk)\cong H^*_{K}(M;\Bbbk)
		\end{tikzcd}
		\]		
		is surjective and $H^{*}_{T}(\Gamma_T;\Bbbk)$ is a free $H^*(BT;\Bbbk)$-module, for $\Bbbk=\Q$. 
		Furthermore, its kernel is equal to $\ker p\cdot H^*_T(\Gamma_T;\Bbbk)$, where $p:H^*(BT)\to H^*(BK)$ is the epimorphism given by the GKM graph extension.
		(Here, $\Bbbk=\Q$, or $\Bbbk=\Z$ provided that all stabilizers of the $K$-action on $M$ are connected.)
    \end{thmM}
	
	The proof of Theorem~\ref{thmM:epi} follows by chasing through a diagram consisting of the ABFP-type sequences for abstract GKM graphs that is similar to the complex $\mathcal{B^{*}(P)}$ introduced in~\cite{fr-16'}.
	Theorem~\ref{thmM:epi} and the results of~\cite{go-so-25} about extensions of GKM graphs to \textit{torus graphs} (i.e. $T$-labeled GKM graphs of valence $\dim T$) imply the second main result of this article.
    (For the definition of the face ring $\Bbbk[S(\Gamma)]$ and the poset $S(\Gamma)$ see~\S\ref{sec:thm2pf} and~\S\ref{ssec:face} below, respectively.)
	
	\begin{thmM}\label{thmM:quot}
		In the situation of Theorem~\ref{thmM:epi}, assume that $M$ is a GKM$_{4}$ $K$-manifold $M$ with the GKM graph $\Gamma=\Gamma_{K}$ that is either Hamiltonian, or has \textit{complexity} $1/2\cdot\dim M-\dim K=1$.
		Then there is an isomorphism of rings 
		\[
		H^{*}_{K}(M;\Bbbk)\cong 
		\Bbbk[S(\Gamma)]/I,
		\]
		for some ideal $I$ generated by degree $2$ elements.
	\end{thmM}
	More precisely, the ideal $I$ is the image of $\ker p\cdot H^{*}_{T}(\Gamma_{T};\Bbbk)$ under the isomorphism with the face ring, where $\Gamma_T$ is a GKM graph extension of $\Gamma_{K}$ to a torus graph (i.e. $\dim T=1/2\cdot\dim M$) and $p\colon H^{*}(BT)\to H^{*}(BK)$ the corresponding epimorphism.
	Theorem~\ref{thmM:quot} is a generalization of~\cite[Theorem~2]{ay-ma-so-23}. Note that Theorem~\ref{thmM:epi} fills a gap in the argument of~\cite[Theorem~2]{ay-ma-so-23}, as the surjectivity of the restriction map (denoted there also by $p_*$) was used in~\cite{ay-ma-so-23} without proof.
	In order to use the well-known proof of the isomorphism with the face ring as in~\cite{ma-pa-06,ma-ma-pa-07} in the case of unsigned torus graphs, we show that they always admit a lift to a $T$-graph (see Proposition~\ref{pr:lifttor}).\\
	
	\noindent {\bf Acknowledgements:} 
	O.G. thanks Leopold Zoller for discussions on the realization problem for equivariant graph cohomology.
	G.S. thanks Matthias Franz for the discussion of syzygies. We gratefully acknowledge funding of the Deutsche Forschungsgemeinschaft
	(DFG, German Research Foundation): Project numbers 561158824 (Walter Benjamin Fellowship of G.S.) and 452427095 (O.G).
	
	\section{Preliminaries}
	
	In this section, we provide the reader with the notions from GKM theory that are necessary for the proof of the main results of this note.
	The contents of this section are well-known.
	
	\subsection{GKM manifolds and GKM graphs}
	
	In this note $(V,E)$ denotes an abstract connected $n$-valent graph without loops.
	For any (oriented) edge $e\in E$ the initial and terminal vertices of $e$ are denoted by $i(e),t(e)\in V$, respectively.
	For any edge $e\in E$ denote by $\overline{e}\in E$ the edge with opposite orientation.
	
	\begin{defn}[\cite{gu-za-01}]\label{defn:abstractsignedgraph}
		An \emph{(abstract) unsigned GKM graph} $\Gamma$ is the tuple of 
		\begin{enumerate}[label=(\roman*)]
			\item an abstract graph $(V,E)$;
			\item a \textit{connection} $\nabla$ on $(V,E)$, i.e. a collection $\lb\nabla_{e}\rb_{e\in E}$ of bijective maps $\nabla_{e}$ from incident edges to $i(e)$ to incident edges to $t(e)$ satisfying $\nabla_{e}e=\overline{e}$ and $\nabla_{\overline{e}}^{-1}=\nabla_{e}$;
			\item an unsigned \emph{axial function} $\alpha:E\to \Z^k/\lb \pm 1\rb$, i.e. a map satisfying $\alpha(\overline{e}) = \alpha(e)$ as well as that for any choice of lifts $\tilde{\alpha}(e')\in \Z^k$ of $\alpha(e')$ and $\tilde{\alpha}(\nabla_e e')$ of $\alpha(\nabla_e e')$ there exists $\eps\in \{\pm 1\}$ such that
			\begin{equation}\label{eq:gkmcomp1}
				\tilde{\alpha}(\nabla_{e}e') \in \eps\tilde{\alpha}(e') + \Z \alpha(e).
			\end{equation}
			for all $e,e'\in E$ with a common origin.
		\end{enumerate}
		Additionally, the $\Z$-span of the image of $\alpha$ is required to be all of $\Z^k/\lb \pm 1\rb$ (this condition is sometimes called \textit{effectivity} of the GKM graph). 
	\end{defn}
	
	One obtains the notion of a \textit{signed GKM graph} by demanding that the axial function $\alpha$ takes values in $\Z^{k}$ in Definition~\ref{defn:abstractsignedgraph}, satisfies the conditions
	\begin{equation}\label{eq:gkmcomp2}
		\alpha(\overline{e}) = -\alpha(e),\ 
		\alpha(\nabla_{e}e') \in \alpha(e') + \Z \alpha(e),
	\end{equation}
	and has span $\Z^{k}$. If $k=n$ in the definition of a GKM graph then we obtain a so-called \textit{(signed or unsigned, respectively) torus graph}.
	
	A (signed or unsigned) GKM graph $\Gamma$ is called a \textit{GKM$_{j}$ graph} if any $i\leq j$ pairwise distinct elements of the multiset $\lb \alpha(e)\mid e\in \str v\rb$, where
    \[
    \str v:=\lb e\in E\mid i(e)=v\rb, 
    \]
    are linearly independent for any $v\in V$. Note that linear independence is a sensible concept for elements in $\Z^k/\{\pm 1\}$.
	
	\begin{rem}
		Oftentimes, one does not include the connection $\nabla$ in the definition of an abstract GKM graph; instead, one assumes only the existence of a connection compatible with the axial function, i.e., satisfying \eqref{eq:gkmcomp1} (respectively \eqref{eq:gkmcomp2}). In this note we focus mainly on studying GKM$_q$ graphs for $q\geq 3$. In this situation the compatible connection is unique.
	\end{rem}
	
	\begin{defn}[\cite{gu-za-01},~\cite{go-ko-zo-22}]\label{defn:gkm}
		Let $M^{2n}$ be a smooth, closed, connected, orientable manifold with a smooth effective action of the compact torus $T^k=(S^1)^k$. 
		The $T^{k}$-action on $M^{2n}$ is called a \textit{GKM manifold} over $\Bbbk$ (recall that $\Bbbk$ is either $\Z$ or $\Q$) if the fixed point set $M^{T^k}$ of the action is finite and non-empty, $M_{1}$ is a finite union of two-dimensional spheres, where 
		\[
		M_{i}:=\{p\in M\mid \dim T\cdot p\leq i\}\subseteq M,\ 
		i\leq k,
		\]
		denotes the \textit{equivariant i-skeleton of $M$}, and $H^{odd}(M;\Bbbk)=0$ holds in singular cohomology.
		In case $n=k$, one obtains a \textit{torus manifold $M$ (e.g. see~\cite{ma-pa-06})} by omitting $H^{odd}(M;\Bbbk)=0$ in this definition.
	\end{defn}
	
	For any GKM $T$-manifold $M$ the orbit space of $M_{1}$ is homeomorphic to a graph $(V,E)$ with vertex set $V=M^{T}$.
	The graph $(V,E)$, equipped with the axial function given by the weights of the isotropy representations in the $T$-fixed points is an unsigned GKM graph $\Gamma=\Gamma(M)$, after identifying the integer lattice ${\mathfrak{t}}_{\Z}^*$ in the dual ${\mathfrak{t}}^*$ of the Lie algebra ${\mathfrak{t}}$ of the torus $T^{k}$ with $\Z^k$. 
	In the following we will identify ${\mathfrak{t}}_{\Z}^*$ with $H^2(BT;\Z)$, and even speak about \emph{abstract} signed (unsigned, respectively) GKM graphs with $T$-labeling, meaning that the axial function takes values in $H^2(BT;\Z)$ (respectively $H^2(BT;\Z)/\{\pm 1\}$). This will be often reflected in the notation by denoting a GKM graph with $T$-labeling as $\Gamma_T$.
	Furthermore, if $M^{2n}$ has a $T$-invariant almost complex structure, then $\Gamma$ lifts to a signed GKM graph.
	For details, in particular an argument for the existence of a compatible connection, see~\cite{gu-za-01} or~\cite[Prop.\ 2.3]{go-wi-22}.
	
	\subsection{Graph equivariant cohomology and extensions} \label{subsec:eqcohomextension}
	
	Let $\Gamma$ be an abstract GKM graph with $T$-labeling. The \textit{equivariant graph cohomology} of $\Gamma$ with coefficients in $\Bbbk$ is the $H^{*}(BT;\Bbbk)\cong\Bbbk[t_{1},\dots,t_{k}]$-algebra
	\begin{equation}\label{eq:gkmring}
		H^{*}_{T}(\Gamma;\Bbbk):=\lb f\colon V\to H^{*}(BT;\Bbbk)\mid f(t(e))-f(i(e))\in (\alpha(e)),\ e\in E\rb,
	\end{equation}
	where $(V,E)$ is the underlying graph of $\Gamma$ and $\alpha$ is its axial function~\cite{gu-za-01}.
	
	\begin{defn}[{\cite{ku-19}}]
		Consider two tori $K\subset T$. By definition, a GKM graph $\Gamma_{T}$ with $T$-labeling \textit{extends} a GKM graph $\Gamma_{K}$ with $K$-labeling if both share the same underlying graphs and connections, whereas the corresponding axial functions $\alpha_T$ and $\alpha_K$ satisfy 
		\begin{equation}\label{eq:extension}
			p\circ \alpha_T=\alpha_K
		\end{equation}
		with $p$ being the induced group epimorphism $p\colon H^{2}(BT;\Z)\to H^{2}(BK;\Z)$ by the embedding $K\subset T$ for signed GKM graphs $\Gamma_{T},\Gamma_{K}$ (respectively its induced map $H^2(BT;\Z)/\{\pm 1\}\to H^2(BK;\Z)/\{\pm 1\}$ for unsigned GKM graphs $\Gamma_{T},\Gamma_{K}$).
	\end{defn}
	
	Such an extension induces a morphism of algebras 
	\[
	p_{*}\colon H^{*}(\Gamma_{T};\Bbbk)\to H^{*}(\Gamma_{K};\Bbbk).
	\]
	
	\subsection{Faces and orbit spaces of GKM manifolds}\label{ssec:face}
	
	For any GKM $T$-manifold $M$ over $\Bbbk$ with $\dim T=k$ denote the natural projection to the orbit space by $\pi\colon M\to Q:=M/T$. 
	Let $Q_{i}:=\pi(M_{i})$, $i\leq k$.
	A \textit{face of the orbit space} is the closure $F$ of any connected component of $Q_{i}\setminus Q_{i-1}$.
	Let $\rk F:=i$ be the \textit{rank} of $F$.
	
	The preimages $M_F:=\pi^{-1}(F)$ of the faces of $Q$ of rank $i$ are exactly the closures of the sets $M_i\setminus M_{i-1}$. These coincide with the connected components $N$ of the fixed point submanifolds $M^H:=\lb x\in M\mid hx=x\ \forall h\in H\rb$, where $H\subset T$ runs over the connected components of the isotropy groups of the $T$-action of codimension $i$. Any such $N$ is $T$-invariant. 
	We call a submanifold $M_F$ a \emph{face submanifold} of $M$.
    Let $T_{F}:=\lb t\in T\mid tx=x\ \forall x\in M_{F} \rb$ be the isotropy subgroup of the $T$-action on $M_{F}$.
	
	\begin{lm}[{\cite[Lemma~2.2]{ma-pa-06}}]\label{lm:efface}
		Any connected component $M_{F}$ as above satisfies $H^{odd}(M_{F};\Bbbk)=0$. The $T/T_{F}$-action on $M_{F}$ is again of GKM type, provided that $\rk F>1$.
	\end{lm}
	
	A connected subgraph $\Theta$ of a GKM graph $\Gamma$ is called an \textit{$i$-face of $\Gamma$} if $\Theta$ is invariant with respect to the connection of $\Gamma$ along edges of $\Theta$ and has valence $i$~\cite{gu-za-01},~\cite{ay-ma-so-23}.
	A face of valence $n-1$ is called a \textit{facet} of $\Gamma$.
	
	The collection $S(M)$ of face submanifolds in $M$ has a partial order given by inclusion of submanifolds.
	The poset $S(M)$ has the rank function given by $1/2\cdot\dim M_F$.
	The collection $S(\Gamma)$ of face subgraphs in $\Gamma$ has a partial order given by inclusion of graphs.
	The poset $S(\Gamma)$ has the rank function given by graph valence.
	Let $S_{i}(\Gamma)$, $S_{i}(M)$ be the subposets of all rank $\leq i$ elements in $S(\Gamma)$, $S(M)$, respectively.
	
	\begin{lm}[{\cite[Proposition~3.9]{so-23}}]\label{lm:gmpos}
		Let $\Gamma$ be the GKM graph of a $\GKM_{j}$-manifold $M$.
		Then $S_{i}(M)=S_{i}(\Gamma)$ for $i<j$.
	\end{lm}
	
	\begin{lm}\label{lm:abspos}
		For any GKM$_{3}$ graph $\Gamma$, $S_{2}(\Gamma)$ is the cell poset of a regular CW complex.
	\end{lm}
	\begin{proof}
		As $M_1/T$ is a graph, it has a natural CW structure, and $S_1(\Gamma)$ is its cell poset. The poset $S_2(\Gamma)$ is the cell poset of the regular CW complex obtained from $M_1/T$ by attaching two-dimensional discs along $2$-faces of $\Gamma$.
	\end{proof}
	
	Every face $F$ of the orbit space $Q$ has boundary $\partial F:=\lb x\in F\colon \dim T_{x}> \dim T-\dim F\rb$, where $T_{x}\subseteq T$ denotes the stabilizer of the orbit $x$. For further reference, we mention a particular case of a statement that was proved more generally for equivariantly formal $T$-spaces.
	
	\begin{pr}[{\cite{ma-pa-06},~\cite[Lemma~1.2]{ay-ma-23},~\cite[Proposition~3.10]{ay-ma-so-23}}]\label{pr:homcell}
		Consider a GKM $T$-manifold $M$.
		Let $\Bbbk=\Q$, or let $\Bbbk=\Z$ if the respective stabilizers are connected. 
		Then for every face $F$ of the orbit space $Q$ the following hold:
		\begin{enumerate}[label=(\roman*)]
			\item $\widetilde{H}^{*}(F;\Bbbk)=0$;
			\item $H^{i}(F,\partial F;\Bbbk)=
			\begin{cases}
				0,\ i\neq\rk F,\\
				\Bbbk,\ i=\rk F.
			\end{cases}
			$
		\end{enumerate}
	\end{pr}
	
	Let $o=\lb o_{p}\rb_{p\in P}$ be any a choice of orientations for all cells in a regular CW complex $X$ with the cell poset $P$.
	Denote by $[\sigma_{p}:\sigma_{q}]_{o}$ the incidence number for the cells $\sigma_{q}\subset \sigma_{p}$, $q<_{1} p$ (i.e. $q<p$ and $\rk p=\rk q+1$), in $X$ with respect to $o$.
	Recall that for any $p<_{2}q$ in $P$ the interval $(p,q)$ consists of two distinct incomparable elements, say $p',q'$~\cite{bj-84}.
	The following lemma gives a characterization of these numbers in an abstract way.
	
	\begin{lm}[{\cite[Ch.~IX, Theorem~7.2]{ma-91}}]\label{lm:cwpos}
		Let $X$ be a regular finite CW complex with cells $\sigma_{p}$, $p\in P$, where $P=P(X)$ is the cell poset of $X$.
		Then the following collections of numbers are equal:
		\begin{enumerate}[label=(\roman*)]
			\item $\lb [\sigma_{p}:\sigma_{q}]_{o}\mid  q<_{1} p\rb$, where $o=\lb o_{p}\rb_{p\in P}$ runs over all choices of orientations for all cells in $X$;
			\item $\lb [r:s]\in \lb\pm 1\rb\mid s<_{1} r\rb$
			satisfying for any $q<_{2} p$ in $P$ the condition
			\begin{equation}\label{eq:diam}
				[p:p'][p',q]+[p:q'][q':q]=0.
			\end{equation}
		\end{enumerate}
	\end{lm}
	
	\section{Atiyah-Bredon-Franz-Puppe sequences for torus actions and GKM graphs}
	
	\subsection{GKM manifold case}\label{ssec:manabfp}
	
	Consider an action of a compact torus $T$ of rank $k$ on a compact connected manifold $M$. 
	We abbreviate 
	\[
	{\overline{AB}}^{i}_{T}(M):=
	H^{*+i}_T(M_{i},M_{i-1}),
	\]
	using the convention that ${\overline{AB}}^0_{T}(M)=H^{*}_{T}(M_0)=H^*_T(M^T)$. Then the (augmented) Atiyah--Bredon--Franz--Puppe sequence (e.g. see~\cite{al-fr-pu-14}) of this action is the following cochain complex:	
	\begin{equation}\label{eq:abfpman}
		\begin{tikzcd}[sep=.3cm]
			0\arrow{r} & 
			H^{*}_{T}(M)\arrow{r} & 
			{\overline{AB}}^{0}_{T}(M)\arrow{r} &
			{\overline{AB}}^{1}_{T}(M)\arrow{r} &
			\cdots\arrow{r} &
			{\overline{AB}}^{k}_{T}(M)\arrow{r} & 0;
		\end{tikzcd}
	\end{equation}
	here, the first map is induced by the inclusion $M^T\to M$, the second is the boundary operator in the long exact sequence of the pair $(M_1,M^T)$, and the other ones are the boundary operators in the long exact sequences of the triples $(M_{i-1},M_i,M_{i+1})$. 
	As the fixed point set of GKM actions is finite, we may identify 
	\[
	{\overline{AB}}^0_T(M) = H^*_T(M^T) = \bigoplus_{p\in M^T} H^*(BT)
	\]
	in the GKM case. We further identify $H^*(BT)$ with the ring of polynomials, either integer polynomials on ${\mathfrak{t}}_\Z$, the integer lattice in the Lie algebra ${\mathfrak{t}}$, in case of integer coefficients, or rational polynomials in case of rational coefficients. We simply write $S({\mathfrak{t}}^*)$ to treat both cases simultaneously.
	
	We recall the following result:
	
	\begin{thm}[{\cite[Theorem~1.1]{fr-pu-07}}]\label{thm:ABFPexact}
		Let $M$ be any GKM $T$-manifold over $\Bbbk=\Q$ or $\Z$.
		In the latter case, additionally assume that all stabilizers of the $T$-action on $M$ are connected.
		Then ${\overline{AB}}^{*}_{T}(M)$ is exact.
	\end{thm}
	
	\begin{pr}[{\cite[proof of Lemma~1.2]{ay-ma-23}}]\label{pr:abmandesc}
		Let $M$ be a GKM$_{j}$ manifold over $\Bbbk=\Q$ or $\Z$. If $\Bbbk=\Z$, assume that all stabilizers of the $T$-action on $M$ are connected. Then one has
		\[
		{\overline{AB}}^i(M)\cong \bigoplus_{F:\, \dim F = i} S({\mathfrak{t}}_F^*),\ 
		i< j
		\]
		and the connecting homomorphism (at a component $F<_{1} G$) in~\eqref{eq:abfpman} coincides with the natural projection
		\[
		S({\mathfrak{t}}_F^*)\to S({\mathfrak{t}}_G^*),
		\]
		up to a sign $[G:F]\in \lb \pm 1\rb$ satisfying~\eqref{eq:diam}.
	\end{pr}
	\begin{proof}
		We have the following chain of isomorphisms
		\begin{align}\label{eq:compab}
			{\overline{AB}}^i(M) & = H^*_T(M_i,M_{i-1}) = \bigoplus_{F:\, \dim F=i} H^*_T(M_F,M_F\cap M_{i-1}) \\
			&\nonumber= \bigoplus_{F:\, \dim F = i} H^*(F,\partial F)\otimes H^*(BT_F) \cong \bigoplus_{F:\, \dim F = i} S({\mathfrak{t}}_F^*).
		\end{align}
		The last two isomorphisms in~\eqref{eq:compab} follow by freeness of the $T/T_{F}$-action on $M_F\setminus M_{i-1}$ and $(F,\partial F)$ being a homology disk by Proposition~\ref{pr:homcell}.
		In particular, we have to make choices of orientations for every homology cell $(F,\partial F)$.
		This implies the first claim of the proposition.
		The second claim follows by a standard computation, using the above made choice of orientations.
		Since~\eqref{eq:abfpman} is a cochain complex, we have~\eqref{eq:diam}, i.e. the last claim of the proposition.
	\end{proof}
	
	\begin{ex}
		Given a GKM manifold $M$, for $i=0,1$, the maps in the ABFP sequence
		\[
		\bigoplus_{v\in V} S({\mathfrak{t}}_v^*) \cong {\overline{AB}}^0(M) \to {\overline{AB}}^{1}(M) \cong \bigoplus_{e\in E}  S({\mathfrak{t}}_e^*)
		\]
		can be understood as follows: a polynomial $f \in S({\mathfrak{t}}_v^*)$ maps to the tuple of all incident to $v$ edges in the graph, whose $e$-component contributes the restriction of $f$ to ${\mathfrak{t}}_e={\mathfrak{t}}_v/(\alpha(e))$, with sign corresponding to whether the orientations of $v$ and $e$ match or not.
	\end{ex} 
	
	\subsection{GKM graph case}
	
	In \S\ref{ssec:manabfp}, we rewrote (up to a cochain isomorphism) the part of the ABFP sequence of a GKM$_{j}$ $T$-manifold up to ${\overline{AB}}^{j-1}(M)$ purely in terms of the GKM graph of the action (and an auxiliary choice of orientation of each face). 	In the sequel we will be mainly interested the GKM$_3$ setting; both in the geometric situation of GKM actions on manifolds as well as the combinatorial situation of abstract GKM graphs. So let us define an ABFP-type sequence for an abstract GKM$_{3}$ graph in small degrees.

	For every face $F$ of the $T$-labeled GKM graph $\Gamma$, let $\mathfrak{t}_{F}:=\spn\la \alpha(e)\mid e\in E_{F}\ra$ be the $\Z$-span of the axial function values on the edges of $F$.
	Denote by $T_{F}\subseteq T$ the subtorus corresponding to $\mathfrak{t}_{F}\subset\mathfrak{t}$ (e.g. defined by the exponential map).
	
	\begin{defn}[{Compare with complex $\mathcal{B^{*}(P)}$~\cite{fr-16'}}]
		Let $\Gamma$ be an abstract GKM$_{3}$ graph.
		Let $\lb [G:F]\mid | F<_{1} G\rb$ be a sign choice for the poset $S_{2}(\Gamma)$ (see Lemmas~\ref{lm:abspos},~\ref{lm:cwpos}).
		Define the (augmented) ABFP-type sequence
		\begin{equation}\label{eq:abfpgra}
			\begin{tikzcd}[sep=.3cm]
				0\arrow{r} & 
				H^{*}_T(\Gamma)\arrow{r} & 
				{\overline{AB}}^{0}(\Gamma)\arrow{r} &
				{\overline{AB}}^{1}(\Gamma)\arrow{r} &
				{\overline{AB}}^{2}(\Gamma),
			\end{tikzcd}
		\end{equation}
		where
		\[
		{\overline{AB}}^{i}(\Gamma)=\bigoplus_{F:\, \dim F = i} H^*(BT_F)=\bigoplus_{F:\, \dim F = i} S({\mathfrak{t}}_F^*),\ 
		i=0,1,2,
		\]
		\[
		d=(d_{F}),\ d_{F}:=\sum_{F<_{1} G} [G:F]\pi(F,G),\ 
		\pi(F,G)\colon S({\mathfrak{t}}_F^*)\to S({\mathfrak{t}}_G^*),
		\]
		and $\pi(F,G)$ is the quotient map induced by the embedding $T_{G}\subset T_{F}$.
	\end{defn}
	
	\begin{lm}\label{lm:ABFPisomorphic}
		The sequence~\eqref{eq:abfpgra} is a cochain complex which does not depend on the choices of signs, up to an isomorphism.
		If $\Gamma=\Gamma(M)$ for a GKM$_{3}$ manifold then the complexes $\overline{AB}^{*}(\Gamma)$ and $\overline{AB}^{*}(M)$ are isomorphic in degrees up to $\overline{AB}^2$.
	\end{lm}
	\begin{proof}
		The first claim follows from
		\[
		\pi(F,F')\pi(F',G)=
		\pi(F,G')\pi(G',G),
		\]
		where $F<F',G'<G$ in $S(\Gamma)$, and~\eqref{eq:diam} immediately.
		The second follows from Lemma~\ref{lm:cwpos}, since the signs are given by orientations of the homology cells by the construction.
		This implies the last claim of the lemma by Proposition~\ref{pr:abmandesc}.
	\end{proof}
	
	\begin{ex}
		Consider an abstract GKM$_{3}$ graph $\Gamma$. We fix an orientation of each edge of $\Gamma$, as well as an orientation of each $2$-face of $\Gamma$ (i.e., a direction in which we traverse the connection path that bounds the face). Then we obtain an ABFP-type sequence for $\Gamma$ of the form
		\[
		\begin{tikzcd}[sep=.3cm]
			0\arrow{r} & 
			H^{*}_T(\Gamma)\arrow{r} & 
			{\overline{AB}}^{0}(\Gamma)\arrow{r} &
			{\overline{AB}}^{1}(\Gamma)\arrow{r} &
			{\overline{AB}}^{2}(\Gamma),
		\end{tikzcd}
		\]
		The first map is the natural inclusion; the second is
		\[
		{\overline{AB}}^0(\Gamma)\to {\overline{AB}}^1(\Gamma);\, (f_v)_{v\in V(\Gamma)} \mapsto (g_e)_{e\in E(\Gamma)},
		\]
		where $g_e = \left.(f_{i(v)}-f_{t(v)})\right|_{{\mathfrak{t}_e}}$, and the third is
		\[
		{\overline{AB}}^1(\Gamma)\to {\overline{AB}}^2(\Gamma);\, (g_e)_{e\in E(\Gamma)} \mapsto (h_F)_{F \textrm{ a }2-\textrm{face}}
		\]
		where $h_F=\left.\sum_{e \textrm{ edge of }F}\pm g_e\right|_{{\mathfrak{t}}_F}$, with the sign in front of $g_e$ being $+1$ if in the chosen orientation of $F$ the edge $e$ is traversed in positive direction, otherwise $-1$.
	\end{ex}
	
	By definition of equivariant graph cohomology, the above ABFP sequence is always exact at $H^*_T(\Gamma)$ (i.e. $H^*_T(\Gamma)$ is torsion-free) and at ${\overline{AB}}^0(\Gamma)$. By~\cite[Theorem~3.2]{al-fr-pu-21} this implies that $H^*_T(\Gamma;\Bbbk)$ is a second syzygy, i.e., a reflexive $H^{*}(BT;\Bbbk)$-module ($\Bbbk=\Q,\Z$). Note that for the proof of this implication in \cite[Theorem~3.2]{al-fr-pu-21} only their Condition $(1)$ is used, which is satisfied for ${\overline{AB}}^i(\Gamma)$ (for $\Bbbk=\Z$ this uses connectedness of the respective stabilizers).
    In the GKM manifold case, reflexivity of $H^*_T(\Gamma;\Bbbk)$ and exactness of the ABFP sequence at the zero term are equivalent by~\cite[Thm.~1.1]{fr-16'}. Alternatively, one may also prove the reflexivity of $H^*_T(\Gamma;\Bbbk)$ for any abstract GKM graph $\Gamma$ for any $\Bbbk=\Q,\Z$ using~\cite[Chap.~VII, \S~4, No.2, Proposition~6~(i), p.519]{bo-72}, as $H^*_T(\Gamma;\Bbbk)$ is the intersection of the free modules 
    \[
	M_{e}:=\lb f\colon V\to H^{*}(BT;\Bbbk)\mid f(t(e))-f(i(e))\in (\alpha(e))\rb.
	\] 
	
	\section{Proof of Theorem~\ref{thmM:epi}}
	
	Let us prove Theorem~\ref{thmM:epi}:
	
	\begin{proof}
		Consider first any codimension $1$ embedding of tori $K\subset T$ defined by a nonzero element $x\in H^2(BT)$. In this situation we have the following commutative diagram
		\[
		\begin{tikzcd}
			&0 \ar[d] & 0 \ar[d] & 0 \ar[d] & 0 \ar[d]\\
			0 \ar[r] & H^*_T(\Gamma_T) \ar[r]\arrow{d}{\cdot x} & \overline{AB}^0(\Gamma_T) \ar[r]\arrow{d}{\cdot x} & \overline{AB}^1(\Gamma_T) \ar[r] \arrow{d}{\cdot x} & \overline{AB}^2(\Gamma_T)\arrow{d}{\cdot x}\\
			0 \ar[r] & H^*_T(\Gamma_T) \ar[r]\arrow {d}{p_*} & \overline{AB}^0(\Gamma_T) \ar[r]\ar[d] & \overline{AB}^1(\Gamma_T) \ar[r] \ar[d] & \overline{AB}^2(\Gamma_T)\ar[d]\\
			0 \ar[r] & H^*_K({\Gamma_K}) \ar[r] & \overline{AB}^0({\Gamma_K}) \ar[r]\ar[d] & \overline{AB}^1({\Gamma_K}) \ar[r] \ar[d] & \overline{AB}^2({\Gamma_K})\ar[d]\\
			& & 0 & 0 & 0.
		\end{tikzcd}
		\]
		whose rows are ABFP sequences. As the $K$-action on $M$ is equivariantly formal and satisfies GKM$_3$, we can identify the bottom row with a portion of the topological ABFP sequence of this action by Lemma~\ref{lm:cwpos}, which is known to be exact by Theorem~\ref{thm:ABFPexact}.  The above two rows are exact at the positions $H^*_T(\Gamma_T)$ and $AB^0(\Gamma_T)$ by definition of equivariant graph cohomology, but a priori exactness at $\overline{AB}^1(\Gamma_T)$ is unclear. The vertical columns of the diagram are exact.
		
		We will show that in this situation the top two rows of the diagram are also exact at $\overline{AB}^1(\Gamma_T)$, and $p_*$ is surjective. As a consequence, the first column in the diagram will turn out to be a short exact sequence. Let us show that this then implies that $H^*_T(\Gamma_T)$ is a free module. To this end, let $\omega_i\in H^*_T(\Gamma_T)$ be elements such that the $p_*\omega_i$ form a $H^*(BK)$-basis of $H^*_K(\Gamma_K)$. We prove that the $\omega_i$ form an $H^*(BT)$-basis of $H^*_T(\Gamma_T)$. To see that the $\omega_i$ are linearly independent, assume that there was a nontrivial linear combination $\sum_{i=1}^n f_i \omega_i=0$. Applying $p_*$ and using that that the $p_*\omega_i$ are linearly independent, it follows that $f_i\circ p = 0$ for all $i$, i.e., $f_i = x f'_i$ for some $f'_i\in H^*(BT)$. Hence, $0 = x\cdot \sum_{i=1}^n f'_i\omega_i$, and as $H^*_T(\Gamma_T)$ is torsion-free, it follows that $\sum_{i=1}^n f'_i\omega_i=0$. As we may repeat this argument, we arrive at a contradiction: after finitely many iterations, some coefficient will have degree zero and cannot be divisible by $x$. The fact that the $\omega_i$ generate is shown by induction on the degree, using that the $p_*\omega_i$ generate. Indeed, one of the $\omega_i$ is necessarily in $H^0_T(\Gamma_T)$ (respectively $\pm 1\in H^0_T(\Gamma_T)$ in case of integer coefficients), as such an element is needed to generate $H^*_K(\Gamma_K)$. For the induction step, given $\eta\in H^*_T(\Gamma_T)$, as the $p_*\omega_i$ generate  $H^*_K(\Gamma_K)$, we may subtract a certain linear combination of the $\omega_i$ from $\eta$ to obtain an element in the kernel of $p_*$. Then write $\eta = x\cdot \eta'$ for some $\eta'\in H^*_T(\Gamma_T)$ and apply the induction hypothesis to $\eta'$.
		
		This proves all statements of the theorem for codimension one extensions. For an arbitrary extension $K\subset T$ we may choose a sequence of codimension one  subtori $K = T_0\subset T_1\subset \cdots \subset T_r = T$ and apply the above argument successively to the extension of GKM graphs corresponding to the embedding $T_i\subset T_{i+1}$ for all $i=0,\dots,r-1$. This is possible because the lowest row of the next diagram is the same as the upper two rows of the previous diagram, hence exact by the previous step.
		
		So it is sufficient to consider the diagram above. We claim exactness of the top two rows at $\overline{AB}^1(\Gamma_T)$, by induction on the degree. Exactness in the lowest degree (i.e., constant polynomials) is clear because in this degree the sequence is the same as that in the lowest degree for $K$ (i.e., the third row). The induction step is a straightforward diagram chase, starting with an element in $\overline{AB}^1(\Gamma_T)$ in the middle row that maps to zero in $\overline{AB}^2(\Gamma_T)$.
		
		Using exactness of the upper two rows of the diagram, another straightforward diagram chase proves that $p_*$ is surjective. The statement on the kernel of $p_*$ follows by exactness of the left vertical short exact sequence (inductively, for a chain of codimension one extensions).	
	\end{proof}
	
	\section{A GKM$_2$ counterexample} \label{sec:GKMcounter}
	
	In this section, cohomology is taken with integer or rational coefficients. Consider the $T^2$-action on the flag manifold ${\mathrm{SU}}(3)/T^2$ by left multiplication. This action is well-known to be of GKM type~\cite{gu-ho-za-06}. Let $\Gamma$ be its GKM graph with respective axial function $\alpha$. It admits an extension to an (unsigned) GKM graph $\tilde{\Gamma}$ with $T^3$-labeling with the respective epimorphism $p\colon H^{2}(BT^{3})\to H^{2}(BT^{2})$, as given in Fig.~\ref{fig:3gr}.
	We use the notion of a Thom class for an edge of a $3$-valent unsigned GKM graph, as defined in~\cite[Definition~2.15]{go-ko-zo-22'}.
	
	\begin{figure}[h]
		\begin{center}
			\begin{tikzpicture}
				\begin{scope}
					\draw[very thick] (1,1) -- ++(-3,0) -- ++(4,-4) -- ++(0,3);
					\draw[ very thick] (2,0) -- ++(-1,1);
					\draw[very thick] (1,-3)--++(0,4);
					\draw[very thick] (1,-3) -- ++(-3,3) -- ++ (4,0);
					\draw[very thick] (1,-3)--++(1,0);
					\draw[very thick] (-2,0) -- ++(0,1);
					\draw[very thick] (-2,1) -- ++(3,0);
					\draw[very thick] (1,1) -- ++(0,-4);
					\draw[very thick] (2,-3) -- ++(0,3);
					\draw[very thick] (2,0) -- ++(-4,0);
					\node at (-2.4,.5){$x$};
					\node at (1.3,-1.2){$x$};
					\node at (2.4,-1.6){$x$};
					\node at (1.8,.7){$z$};
					\node at (.3,-.85){$z$};
					\node at (-.6,-1.8){$z$};
					\node at (1.5,-3.35) {$y$};
					\node at (-.5,1.3) {$y$};
					\node at (-.2,.25) {$y$};
					\node at (1,1)[circle,fill,inner sep=2pt]{};

					\node at (-2,1)[circle,fill,inner sep=2pt]{};

					\node at (1,-3)[circle,fill,inner sep=2pt]{};

					\node at (-2,0)[circle,fill,inner sep=2pt]{};

					\node at (2,0)[circle,fill,inner sep=2pt]{};

					\node at (2,-3)[circle,fill,inner sep=2pt]{};
					
					\node at (-2.3,1.2){$1$};
					\node at (1.3,1.2){$2$};
					\node at (2.3,-0.2){$3$};
					\node at (2.3,-3.2){$4$};
					\node at (.7,-3.2){$5$};
					\node at (-2.3,-0.2){$6$};
				\end{scope}

			\end{tikzpicture}
			\caption{An unsigned torus graph extending the GKM graph of the flag manifold ${\mathrm{SU}}(3)/T^2$.}
			\label{fig:3gr}
		\end{center}
	\end{figure}
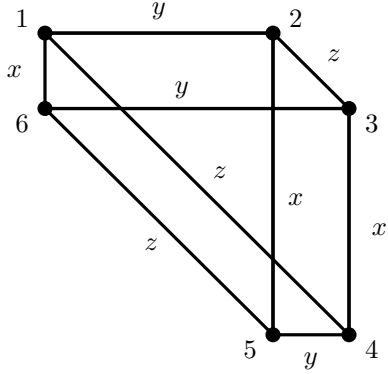
	
	\begin{pr}
		$H^*_{T^3}(\tilde{\Gamma})$ is not a free $H^*(BT^3)$-module. Further, the natural projection $p_*:H^*_{T^3}(\tilde{\Gamma})\to H^*_{T^2}(\Gamma)$ is not surjective.
	\end{pr}
	\begin{proof}
		We enumerate the vertices of $\Gamma$ (respectively $\tilde{\Gamma}$) clockwise, as in the figure above.   Let us first show that $H^2_{T^3}(\tilde{\Gamma})=H^2(BT^3)\cdot 1$. To this end, let $\omega\in H^2_{T^3}(\tilde{\Gamma})$ be arbitrary; by subtracting an appropriate element in $H^2(BT^3)$ we may assume that $\omega$ vanishes at vertex $1$. Considering now the congruence relations along the path $1-2-5-6-1$, it follows that $\omega$ vanishes also at vertex $2$. Following the path $1-2-3-6-1$, we see that $\omega$ vanishes also at vertex $6$. Continuing in the same way, it follows that $\omega=0$.
		
		This already implies that $p_*$ cannot be surjective: if it was, we had a short exact sequence
		\[
		0 \longrightarrow H^*_{T^3}(\tilde{\Gamma})\longrightarrow H^*_{T^3}(\tilde{\Gamma})\overset{p_*}\longrightarrow H^*_{T^2}(\Gamma)\longrightarrow 0,
		\]
		implying that the nonequivariant graph cohomology of $\tilde{\Gamma}$ was equal to
		\[
		H^*_{T^3}(\tilde{\Gamma})/H^+(BT^3)\cdot H^*_{T^3}(\tilde{\Gamma}) = H^*_{T^2}(\Gamma)/H^+(BT^2)\cdot H^*_{T^2}(\Gamma) = H^*_{T^2}({\mathrm{SU}}(3)/T^2) = H^*({\mathrm{SU}}(3)/T^2).
		\]
		But we just observed that this quotient vanishes in degree $2$, while the cohomology of the flag manifold $H^*({\mathrm{SU}}(3)/T^2)$ is nonzero in degree $2$.
		
		It remains to show that $H^*_{T^3}(\tilde{\Gamma})$ is not a free module over $H^*(BT^3)$. We just computed that in degree $2$ there do not occur any new module generators. Assuming that $H^*_{T^3}(\tilde{\Gamma})$ was free, it would follow that any set of elements of $H^4_{T^3}(\tilde{\Gamma})$ that induces a $\Q$-basis of $H^4_{T^3}(\tilde{\Gamma})/H^4(BT^3)\cdot 1$ is also $H^*(BT^3)$-linearly independent. But consider for instance the four Thom classes of the edges~\cite[Def.~2.15]{go-ko-zo-22'} forming the quadrangle $1-2-3-6-1$: e.g., ${\mathrm{Th}}_{12}$ is (with the appropriate sign chosen) the element in $H^4_{T^3}(\tilde{\Gamma})$ which equals $xz$ at vertices $1$ and $2$, and is zero at all other vertices of the graph.
		These four elements are $\Q$-linearly independent in $H^4_{T^3}(\tilde{\Gamma})$, however they satisfy the following relation
		\[
		y\cdot {\mathrm{Th}}_{12} - z \cdot {\mathrm{Th}}_{23} + y \cdot {\mathrm{Th}}_{36} - x \cdot {\mathrm{Th}}_{16} = 0.
		\]
	\end{proof}
	
	\begin{rem}
		The signed GKM graph of the flag manifold $SU(3)/T^{2}$ does not admit an extension to a signed GKM graph (e.g., see~\cite[Example~2.6, Corollary~2.12]{so-23}).
		There is no contradiction with the existence of an extension to an unsigned torus graph, since the respective axial function groups (whose ranks are complete obstructions to GKM graph extendibility~\cite{ku-19, go-so-25}) are not isomorphic, in particular, because they are defined for different connections.
	\end{rem}
	
	\section{Proof of Theorem~\ref{thmM:quot}}\label{sec:thm2pf}
	
	In this section we prove Theorem~\ref{thmM:quot}. 
	
	\subsection{$T$-graphs}	
	\begin{defn}[{\cite{ma-ma-pa-07}}]\label{defn:abstractunsignedgraph}\label{def:tgraph}
		A $T$-graph is a tuple $\Gamma=(V,E,\nabla,\beta)$ consisting of an abstract graph $(V,E)$ as above, a connection $\nabla$ on $(V,E)$ and an \emph{axial function} $\beta\colon E\to \Z^{n}$, where $n$ is the valence of the graph, that satisfy:
		\begin{enumerate}[label=(\roman*)]
			\item For each $e\in E$, $\beta(\overline{e}) = -\delta(e)\cdot\beta(e)$, $\delta(e)\in \lb\pm 1\rb$;
			\item For each $v\in V$ the set $\lb \beta(e)\mid \in\str v\rb$ is a basis of the whole $\Z^{n}$;
			\item The axial function $\beta$ satisfies the \emph{congruence relation} 
			\[
			\beta(\nabla_e e') = \beta(e') + d_{e,e'} \cdot \beta(e).
			\]
		\end{enumerate}
	\end{defn}
	
	Clearly, any $T$-graph structure on a graph $\Gamma$ induces the structure of unsigned torus graph on $\Gamma$, by considering the axial function modulo signs. In the next section we show that conversely, the axial function of any unsigned torus graph lifts to that of a $T$-graph.
	
	\subsection{Lifts to $T$-graphs}	
	
	Recall that a ranked poset $P$ is called \textit{simplicial} if $P_{\leq p}=\lb q\in P\colon q\leq p\rb$ is the lattice of faces in a simplex of dimension $\rk p$.
	A \textit{characteristic function} on a simplicial poset $P$ of dimension $n$ is a function $\lambda\colon P_{0}\to\Z^{n}$ (where $P_{0}$ denotes the subset of rank $0$ elements in $P$) such that for every $p\in P$ the values of $\lambda$ on $(P_{\leq p})_{0}$ form a basis of split summand in $\Z^{n}$ of rank $\rk p+1$.
	
	For any unsigned torus graph $\Gamma$, the opposite poset $S(\Gamma)^{op}$ to $S(\Gamma)$ is simplicial.	
	
	\begin{pr}\label{pr:torchar}
		For a fixed connection $\nabla$ on an $n$-valent graph $\Gamma = (V,E)$ with face poset $S$, there is a bijection between the following sets:
		\begin{enumerate}[label=(\roman*)]
			\item functions $\mathcal{F}(\Gamma)\to\Z^{n}$
			from the set of facets $\mathcal{F}(\Gamma)$ in $\Gamma$ such that for every vertex $v\in V$ of $\Gamma$, the collection $\lb\lambda(F)\mid  F\in \mathcal{F}(\Gamma),\ v\in F\rb$ is a basis of the whole $\Z^{n}$ (this condition is called unimodularity of the function);
			\item characteristic functions on $S^{op}$;
			\item $T$-graphs with face poset $S$.
		\end{enumerate}
	\end{pr}
	\begin{proof}
		It is straightforward to establish the bijection between $(i)$ and $(ii)$.
		Let $\Gamma$ be any $T$-graph with the connection $\nabla$ and the underlying graph $(V,E)$.
		For every vertex $v\in V$, let $\lb\lambda_{v}(e)\mid e\in \str v\rb$ be the basis dual to $\lb\beta(e)\mid e\in \str v\rb$. We wish to show that for any facet $F$, the value $\lambda_v(e)$, where $v\in V_F$ and $e$ is the unique edge transversal to $F$ at $v$, does not depend on $v$. Having achieved this, we may define $\lambda(F):=\lambda_{v}(e)$. Unimodularity of this function is then obvious, as at a vertex $v$ it is given by the dual basis of a basis of $\Z^n$.
		
		To show the claim, we consider any edge $e\in E$ and denote by $e_1=e,e_2,\ldots,e_n$ the edges at $i(e)$ and by $f_1:=\overline{e}$ and $f_j=:\nabla_e e_j$, $j=2,\ldots,n$ the edges at $t(e)$. We need to show that $\lambda_{i(e)}(e_j) = \lambda_{t(e)}(f_j)$ for $j=2,\ldots,n$. But this follows directly from the congruence relation $\beta(f_j) \equiv \beta(e_j) \mod \beta(e_1)$.
		
		Conversely, let $\lambda$ be a characteristic function on $S^{op}=S(\Gamma)^{op}$; we use again the notation $\lambda_v(e) = \lambda(F)$ as above. For $v\in V$ we define $\{\beta(e)\mid e\in \str v\}$ to be the dual basis of $\{\lambda_v(e)\}$ and show that $\beta$ defines an axial function of a $T$-graph structure on $\Gamma$. We need to check conditions $(i)$ and $(iii)$ from Definition~\ref{def:tgraph}. To this end, fix any (oriented) edge $e$, call $e_1:=e$ and $e_2,\ldots,e_n$ the remaining edges at $i(e)$. Further, let $f_1:=\overline{e}$, and $f_j:=\nabla_e e_j$, $j=2,\ldots,n$. By construction, $\lambda_{i(e)}(e_j) = \lambda_{t(e)}(f_j)$ for $j=2,\ldots,n$, and hence $\lambda_{t(e)}(f_1) = \pm \lambda_{i(e)}(e_1) + \sum_{j=2}^n a_j \lambda_{i(e)}(e_j)$ for some $a_j\in \Z$. Comparing the dual bases of $\{\lambda_{i(e)}(e_j)\}$ and $\{\lambda_{t(e)}(f_j)\}$, we arrive immediately at conditions $(i)$ and $(iii)$. The proof is complete.
	\end{proof}
	
	\begin{pr}\label{pr:lifttor}
		Any unsigned torus graph admits a lift to a $T$-graph.
	\end{pr}
	\begin{proof}
		We construct a characteristic function on $S(\Gamma)^{op}$ as follows.
		Lift the axial function $\alpha$ of $\Gamma$ arbitrarily to $\tilde\alpha\colon E \to \Z^{n}$ satisfying $\pi\circ \tilde\alpha = \alpha$.
		For any $v\in V$, let $\{\lambda_{v}(e)\mid e\in \str v\}$ be the dual basis to the $\{\tilde\alpha(e)\mid e\in \str v\}$. Now define 
		\[
		\lambda(F):=\lambda_{v}(e),
		\]
		where for any facet $F$ we chose arbitrarily a vertex $v\in V_{F}$, and $e$ is the unique transversal edge to $F$ at $v$.
		This choice is non-canonical; but note that $\lambda(F)$ is defined canonically up to sign.
		
		Clearly, the unimodularity condition does not depend on the sign choice that we made, and holds true as at every vertex we took the dual to a basis of $\Z^n$. By Proposition~\ref{pr:torchar}, this characteristic function induces a $T$-graph structure on $\Gamma$, which lifts the original structure of GKM graph, as the operation of taking the dual basis is self-inverse.
	\end{proof}
	
	\subsection{Thom classes and the proof}
	
	Let $\widetilde{\Gamma}$ be any $T$-graph with the axial function $\beta$ taking values in $H^{2}(BT;\Z)$.
	The \textit{face ring} of the simplicial poset $S(\widetilde{\Gamma})$ is defined as the quotient of the polynomial ring by the ideal of \textit{straightening relations} (see~\cite{ma-pa-06})
	\[
	\Bbbk[S(\widetilde{\Gamma})]:=\Bbbk[\tau_{F}\mid F\in S(\widetilde{\Gamma})]/(\tau_{F}\cdot\tau_{G}=\tau_{F\vee G}\cdot\sum_{H\in F\wedge G}\tau_{H},\ \tau_\Gamma=1),\ \deg \tau_{F}=2\cdot\rk F,
	\]
	where $F\wedge G$ denotes the collection of connected components for the graph $F\cap G$, and $F\vee G\in S(\widetilde{\Gamma})$ is the least face containing both $F$ and $G$.
	There is the morphism
	\begin{equation}\label{eq:facemap}
		\Bbbk[S(\widetilde{\Gamma})]\to H^{*}_{T}(\widetilde{\Gamma}),\ 
		\tau_{F}\mapsto \biggl[v\mapsto 
		\begin{cases}
			\prod\limits_{i(e)=v,\ e\notin F}{\beta(e)},\ v\in F,\\
			0,\ \mbox{otherwise}.
		\end{cases}\biggr],\ 
        F\in S(\widetilde{\Gamma}).
	\end{equation}
	
	Given an unsigned GKM graph of complexity $0$, one defines its Thom classes as those for a lift to $T$-graphs (by Proposition~\ref{pr:lifttor}).
	Notice that the homomorphism \eqref{eq:facemap} (as well as its domain and codomain) depends only on the unsigned GKM graph of $\widetilde{\Gamma}$.
	
	\begin{thm}[{\cite[Theorem~5.5]{ma-ma-pa-07}}]\label{thm:mp}
		For any $T$-graph $\widetilde{\Gamma}$ the morphism \eqref{eq:facemap} is an isomorphism of algebras over $\Bbbk=\Q,\Z$.
	\end{thm}
	
	\begin{thm}[{\cite{go-so-25}}]\label{thm:goso}
		Let $\Gamma_{K}$ be the GKM graph of any GKM$_{4}$ $K$-manifold $M$.
		Suppose that the $K$-action on $M$ is either Hamiltonian, or has complexity $1$.
		Then $\Gamma_{K}$ extends to a torus graph $\Gamma_{T}$ by an epimorphism $p\colon H^{*}(BT)\to H^{*}(BK)$.
	\end{thm}	

    \begin{proof}[Proof of Theorem~\ref{thmM:quot}]
	By Theorem~\ref{thm:goso} there exists an extension of $\Gamma=\Gamma_{K}$ to an unsigned torus graph $\Gamma_{T}$ with respect to some epimorphism $p$.
    Theorem~\ref{thmM:epi} implies that the induced morphism of graph equivariant cohomology is surjective with kernel spanned by $\ker p$.
    By Proposition~\ref{pr:lifttor}, let $\widetilde{\Gamma}_{T}$ be a $T$-graph lifting the unsigned torus graph $\Gamma_{T}$.
    Then (in particular, by Theorem~\ref{thm:mp}) we have a composition
    \[
    \begin{tikzcd}
    \Bbbk[S(\widetilde{\Gamma}_{T})]\arrow{r}{\cong} &
    H^{*}_{T}(\widetilde{\Gamma}_{T};\Bbbk)\arrow[equal]{r} &
    H^{*}_{T}(\Gamma_{T};\Bbbk)\arrow{r}{p_{*}} &
    H^{*}_{K}(\Gamma_{K};\Bbbk).
    \end{tikzcd}
    \]
    Notice that $S(\widetilde{\Gamma}_{T})=S(\Gamma_{T})=S(\Gamma_{K})$, since the connections of $\Gamma_{K},\Gamma_{T}$ are equal by the definition of a GKM graph extension.
    This completes the proof.
    \end{proof}

\end{document}